\documentclass[12pt,reqno]{amsart}
\usepackage[utf8]{inputenc}
\usepackage[T1]{fontenc}% Pour pouvoir taper les accent directement et non pas passer par
\usepackage[usenames, dvipsnames]{color}
\usepackage{ulem}

\usepackage{dsfont, amsfonts, amsmath, amssymb,amscd, stmaryrd, latexsym, amsthm, dsfont}
\usepackage[frenchb,english]{babel}
\usepackage{enumerate}
\usepackage{longtable}
\usepackage{geometry}
\usepackage{float}
\usepackage{tikz}
\usetikzlibrary{shapes,arrows}
\usepackage{float}
\usepackage{tikz}
\usepackage{xypic}
\usetikzlibrary{shapes,arrows}

\newtheorem{theorem}{Theorem}[section]
\newtheorem{lemma}[theorem]{Lemma}

\newtheorem{corollary}[theorem]{Corollary}

\theoremstyle{remark}
\newtheorem{remark}[theorem]{\bf Remark}

\usepackage{hyperref}
\hypersetup{
	colorlinks=true,
	urlcolor=blue,
	citecolor=blue}
\def\NN{\mathds{N}}
\def\RR{\mathbb{R}}
\def\QQ{\mathbb{Q}}
\def\CC{\mathbb{C}}
\def\ZZ{\mathbb{Z}}
\def\kk{\mathds{k}}
\begin{document}
	
\def\NN{\mathbb{N}}
\def\RR{\mathds{R}}
\def\HH{I\!\! H}
\def\QQ{\mathbb{Q}}
\def\CC{\mathds{C}}
\def\ZZ{\mathbb{Z}}
\def\DD{\mathds{D}}
\def\OO{\mathcal{O}}
\def\kk{\mathds{k}}
\def\KK{\mathbb{K}}
\def\ho{\mathcal{H}_0^{\frac{h(d)}{2}}}
\def\LL{\mathbb{L}}
\def\L{\mathds{k}_2^{(2)}}
\def\M{\mathds{k}_2^{(1)}}
\def\k{\mathds{k}^{(*)}}
\def\l{\mathds{L}}

\selectlanguage{english}
\title[Unit groups of some multiquadratic number fields]{Unit groups of some multiquadratic number fields and  $2$-class groups}
%premier auteur

\author[M. M. Chems-Eddin]{Mohamed Mahmoud Chems-Eddin}
\address{Mohamed Mahmoud CHEMS-EDDIN: Mohammed First University, Mathematics Department, Sciences Faculty, Oujda, Morocco }
\email{2m.chemseddin@gmail.com}


\begin{abstract}
Let $p\equiv -q \equiv 5\pmod 8$  be two prime integers.
In this paper, we investigate  the unit  groups of the fields  $ L_1  =\QQ(\sqrt 2, \sqrt{p}, \sqrt{q}, \sqrt{-1} )$ and $ L_1^+=\QQ(\sqrt 2, \sqrt{p}, \sqrt{q} )$. Furthermore , we  give the second $ 2$-class groups  of the subextensions of $L_1$ as well the $2$-class groups of the fields     $ L_n  =\QQ( \sqrt{p}, \sqrt{q}, \zeta_{2^{n+2}} )$ and    their    maximal real subfelds.
\end{abstract}

\selectlanguage{english}

\maketitle

 \section{Introduction}
 \label{Sec:1}
This paper was written to commemorate  the innocent victims of coronavirus disease (COVID-19) pandemic all around the world. Let $k$ be a number field and $E_k$ its unit group. The determination of  $E_k$
  is a very difficult computational problem   that serves to give answers on many
 problems such as the computation of the class number of $k$, the capitulation problem and many other problems in algebraic number theory. The most spectacular result that describes the structure of
 $E_k$ is the well known  Dirichlet's unit theorem that says  that
 $$E_k=\mu(k)\times \mathbb{Z}^{r_1+r_2-1},$$
 where $\mu(k)$ is the group of roots of unity contained in $k$,   $r_1$ is the number of real embeddings and $r_2$ the number of conjugate pairs of complex embeddings of $k$.
 This is the only known  and general result  that covers any given number field $k$. If $k$  is an  imaginary J-field   there is a known result of Hasse that
 gives the difference between the unit group of $k$ and that of its real maximal subfield $k^+$ i.e.,  the index
 $[E_k:\mu(k)E_{k^+}]$ equals $1$ or $2$.

Unfortunately,   these results  do not give much information  on the generators of the group $E_k$.  For  the particular  family of multiquadratic  number fields there are some useful algorithms by Wada (cf. \cite{wada}) and Azizi
 (cf. \cite{azizunints99}) that helped to compute the unit groups of many families of
 real biquadratic number fields and imaginary triquadratic number fields (cf. \cite{AZT2016,Be05}). Whereas these algorithms became very difficult  to be applied to real
 multiquadratic fields of degree $\geq 8$  and  imaginary multiquadratic fields of degree $\geq16$.
 In the best of our knowledge there is  only one example, in literature,  that explicitly determines the unit groups of some infinite families of    such fields (see the recently published paper
 \cite{chemszekhniniaziziUnits1}). In \S \ref{section 2} of this  paper we shall reinforce the algorithms of Wada and Azizi by a process of elimination based on norm maps and class number formulas
 to explicitly determine  unit groups of  the fields $ L_1  =\QQ(\sqrt 2, \sqrt{p}, \sqrt{q}, \sqrt{-1} )$ and $ L_1^+=\QQ(\sqrt 2, \sqrt{p}, \sqrt{q} )$, where $p$ and $q$ are two primes
 that satisfy  one of the following conditions:

  	\begin{eqnarray}\label{cond 1}
  p\equiv   5\pmod 8, \;q\equiv  3\pmod 8	  \text{ and }  \left(\dfrac{p}{q}\right)=1 ,
  \end{eqnarray}
  \begin{eqnarray}\label{cond 2}
  p\equiv   5\pmod 8, \;q\equiv  3\pmod 8	  \text{ and }  \left(\dfrac{p}{q}\right)=-1.
  \end{eqnarray}	

In \S \ref{section 3}, we  determine the $2$-class groups  and the second $2$-class groups of the unramified quadratic extensions of $\mathbb{Q}(\sqrt{2pq},i)$, as well we give the  $2$-class groups of the layers of their cyclotomic $\ZZ_2$-extension.% $\mathbb{Q}(\sqrt{p}, \sqrt{q})$ and $\mathbb{Q}(\sqrt{p}, \sqrt{q}, i)$.

 \section*{Notations}
 Let $k$ be a number field. We shall respect  the following notations for the rest of this paper:

  \begin{enumerate}[$*$]
 	\item $h_2(k)$: The $2$-class number of $k$,
 	\item $h_2(d)$: The $2$-class number of the quadratic field $\mathbb{Q}(\sqrt{d})$,
 	\item $\varepsilon_d$: The fundamental unit of the quadratic field $\mathbb{Q}(\sqrt{d})$,
 	\item $E_k$: The unit group of $k$,
 	\item FSU: Abbreviation of ``fundamental system of units'',
 	\item $k^{(1)}$: The Hilbert $2$-class field of $k$,
 	\item $k^{(2)}$: The Hilbert $2$-class field of $k^{(1)}$,	
 	\item $k^{+}$: The maximal real  subfield of $k$, whenever $k$ is imaginary,
 		\item $q(k)=(E_{k}: \prod_{i}E_{k_i})$ is the unit index of $k$, if $k$ is multiquadratic, where   $k_i$ are  the  quadratic subfields	of $k$,
 	 	\item $N_{k'/k}$: The norm map of an extension $k'/k$.
 \end{enumerate}

\section{\textbf{Units of some mutiquadratic number fields of degree $8$ and $16$}}\label{section 2}
Let us start by collecting some results that will be useful in the sequel.	
\begin{lemma}[\text{\cite[Lemma 5]{Az-00}}]\label{a5}
 	Let $d>1$ be a square-free integer and $\varepsilon_d=x+y\sqrt d$,
 	where $x$, $y$ are  integers or semi-integers. If $N(\varepsilon_d)=1$, then $2(x+1)$, $2(x-1)$, $2d(x+1)$ and
 	$2d(x-1)$ are not squares in  $\QQ$.
 \end{lemma}

  \begin{lemma}[\cite{azizunints99}, Proposition 2]\label{Lemme azizi} Let $K_0$ be a real number field, $K=K_0(i)$ a quadratic extension of $K_0$, $n\geq 2$ be an integer and $\xi_n$ a $2^n$-th primitive root of unity, then
 	$	\xi_n=\frac{1}{2}(\mu_n+\lambda_ni)$, where $\mu_n=\sqrt{2+\mu_{n-1}}$, $\lambda_n=\sqrt{2-\mu_{n-1}}$, $\mu_2=0$, $\lambda_2=2$ and $\mu_3=\lambda_3=\sqrt{2}$. Let $n_0$ be the greatest
 	integer such that $\xi_{n_0}$ is contained in $K$, $\{\varepsilon_1,...,\varepsilon_r\}$ a fundamental system of units of $K_0$ and $\varepsilon$ a unit of $K_0$ such that
 	$(2+\mu_{n_0})\varepsilon$ is a square in $K_0$(if it exists). Then a fundamental system of units of $K$ is one of the following systems :
 	\begin{enumerate}[\rm 1.]
 		\item $\{\varepsilon_1,...,\varepsilon_{r-1},\sqrt{\xi_{n_0}\varepsilon } \}$ if $\varepsilon$ exists, in this case $\varepsilon=\varepsilon_1^{j_1}...\varepsilon_{r-1}^{j_1}\varepsilon_r$,
 		where $j_i\in \{0,1\}$.
 		\item $\{\varepsilon_1,...,\varepsilon_r \}$ else.
 		
 	\end{enumerate}
 \end{lemma}
	
	Let us recall the method given in    \cite{wada}, that describes a fundamental system  of units of a real  multiquadratic field $K_0$. Let  $\sigma_1$ and 
	$\sigma_2$ be two distinct elements of order $2$ of the Galois group of $K_0/\mathbb{Q}$. Let $K_1$, $K_2$ and $K_3$ be the three subextensions of $K_0$ invariant by  $\sigma_1$,
	$\sigma_2$ and $\sigma_3= \sigma_1\sigma_3$, respectively. Let $\varepsilon$ denote a unit of $K_0$. Then \label{algo wada}
	$$\varepsilon^2=\varepsilon\varepsilon^{\sigma_1}  \varepsilon\varepsilon^{\sigma_2}(\varepsilon^{\sigma_1}\varepsilon^{\sigma_2})^2,$$
	and we have, $\varepsilon\varepsilon^{\sigma_1}\in E_{K_1}$, $\varepsilon\varepsilon^{\sigma_2}\in E_{K_2}$  and $\varepsilon^{\sigma_1}\varepsilon^{\sigma_2}\in E_{K_3}$.
	It follows that the unit group of $K_0$  
	is generated by the elements of  $E_{K_1}$, $E_{K_2}$ and $E_{K_3}$, and the square roots of elements of   $E_{K_1}E_{K_2}E_{K_3}$ which are perfect squares in $K_0$.

 \noindent Let us continue by stating  the following  results.
 \begin{lemma}\label{lm expressions of units under cond 1}
 	Let $p$ and $q$ be two primes satisfying    conditions $(\ref{cond 1})$.
 	\begin{enumerate}[\rm 1.]
 		\item Let  $x$ and $y$   be two integers such that
 		$ \varepsilon_{2pq}=x+y\sqrt{2pq}$. Then we have
 		\begin{enumerate}[\rm i.]
 			\item $p(x-1)$ is a square in $\NN$,
 			\item $\sqrt{2\varepsilon_{2pq}}=y_1\sqrt{p}+y_2\sqrt{2q}$ and 	$2=2qy_2^2-py_1^2$, for some integers $y_1$ and $y_2$.
 		\end{enumerate}
 		% $p(x-1)$ is a square in $\NN$, $\sqrt{2\varepsilon_{2pq}}=y_1\sqrt{p}+y_2\sqrt{2q}$ and 	$2=2qy_2^2-py_1^2$, for some integers $y_1$ and $y_2$   such that $y=y_1y_2$.
 		
 		\item  Let    $a$ and $b$ be two integers such that
 		$ \varepsilon_{pq}=a+b\sqrt{pq}$. Then  we have
 		\begin{enumerate}[\rm i.]
 			\item $2p(a+1)$ is a square in $\NN$,
 			\item   $\sqrt{ \varepsilon_{ pq}}=b_1\sqrt{p}+b_2\sqrt{q}$ and 	$1=pb_1^2-qb_2^2$, for some integers $b_1$ and $b_2$.
 		\end{enumerate}
 		% $2p(a+1)$ is a square in $\NN$, $b$ is even,  $\sqrt{ \varepsilon_{ pq}}=b_1\sqrt{p}+b_2\sqrt{p}$ and 	$1=pb_1^2-qb_2^2$, for some integers $b_1$ and $b_2$   such that $b=2b_1b_2$.
 		
 		\item  Let    $c $ and $d$ be two integers such that
 		$ \varepsilon_{2q}=c +d\sqrt{2q}$. Then  we have
 		\begin{enumerate}[\rm i.]
 			\item   $c-1$ is a square in $\NN$,
 			\item  $\sqrt{ 2\varepsilon_{  2q}}=d_1 +d_2\sqrt{2q}$ and 	$2=-d_1^2+2qd_2^2$, for some integers $d_1$ and $d_2$.
 		\end{enumerate}

 		\item  Let    $\alpha $ and $\beta$ be two integers such that
 		$ \varepsilon_{q}=\alpha +\beta\sqrt{q}$. Then  we have
 		\begin{enumerate}[\rm i.]
 			\item   $\alpha-1$ is a square in $\NN$,
 			\item  $\sqrt{ 2\varepsilon_{  q}}=\beta_1 +\beta_2\sqrt{q}$ and 	$2=-\beta_1^2+q\beta_2^2$, for some integers $\beta_1$ and $\beta_2$.
 		\end{enumerate}
 	\end{enumerate}
 \end{lemma}	
 \begin{proof}~\
\begin{enumerate}[\rm 1.]
 		\item
 		It is known that $N(\varepsilon_{2pq})=1$. Then, by the unique factorization in $\mathbb{Z}$ and Lemma \ref{a5}  there exist some integers $y_1$ and $y_2$  $(y=y_1y_2)$ such that
 		$$(1):\ \left\{ \begin{array}{ll}
 		x\pm1=y_1^2\\
 		x\mp1=2pqy_2^2,
 		\end{array}\right. \quad
 		(2):\ \left\{ \begin{array}{ll}
 		x\pm1=py_1^2\\
 		x\mp1=2qy_2^2,
 		\end{array}\right.\quad
 		\text{ or }\quad
 		(3):\ \left\{ \begin{array}{ll}
 		x\pm1=2py_1^2\\
 		x\mp1=qy_2^2,
 		\end{array}\right.
 		$$		
 		\begin{enumerate}[\rm$*$]
 			\item System $(1)$ can not occur since it implies  
 			$1=\left(\frac{y_1^2}{p}\right)=\left(\frac{x\pm 1}{p}\right)=\left(\frac{x\mp1\pm 2}{p}\right)=\left(\frac{\pm2}{p}\right)=\left(\frac{2}{p}\right)=-1$, which is absurd. 
 			\item Similarly system $(3)$ can not occur too since it implies 
 			$1=\left(\frac{q}{p}\right)=\left(\frac{x\mp 1}{p}\right)=\left(\frac{\pm2}{p}\right)=\left(\frac{2}{p}\right)=-1$, which is absurd.
 			\item Suppose that $\left\{ \begin{array}{ll}
 			x+1=py_1^2\\
 			x-1=2qy_2^2.
 			\end{array}\right.$
 			Then,	$1=\left(\frac{py_1^2}{q}\right)=\left(\frac{x+ 1}{q}\right)=\left(\frac{x-1+ 2}{q}\right)=\left(\frac{2}{q}\right)=-1.$ Which is also impossible.
 		\end{enumerate}	
 		Thus, the only possible case is
 		$\left\{ \begin{array}{ll}
 		x-1=py_1^2\\
 		x+1=2qy_2^2,
 		\end{array}\right.$	
 		which implies that\\ $\sqrt{2\varepsilon_{2pq}}=y_1\sqrt{p}+y_2\sqrt{2q}$ and 	$2=2qy_2^2-py_1^2$.
 		
 		\item As  $N(\varepsilon_{pq})=1$. Then, by Lemma \ref{a5}  we have
 		$$(1):\ \left\{ \begin{array}{ll}
 		a\pm1=pb_1^2\\
 		a\mp1=  qb_2^2,
 		\end{array}\right. \quad
 		(2):\ \left\{ \begin{array}{ll}
 		a\pm1=b_1^2\\
 		a\mp1=pqb_2^2,
 		\end{array}\right.\quad
 		\text{ or }\quad
 		(3):\ \left\{ \begin{array}{ll}
 		a\pm1=2pb_1^2\\
 		a\mp1=2qb_2^2,
 		\end{array}\right.
 		$$	
 		For some integers $b_1$ and $b_2$   such that $b=b_1b_2$ or $b=2b_1b_2$ ($b=2b_1b_2$ in the cases of system  $(3)$). As above we show that the only possible case is 
 		$\left\{ \begin{array}{ll}
 		a+1=2pb_1^2\\
 		a-1=2qb_2^2.
 		\end{array}\right.$ From which we infer that   $\sqrt{ \varepsilon_{ pq}}=b_1\sqrt{p}+b_2\sqrt{q}$ and 	$1=pb_1^2-qb_2^2$.
 		
 		\item As  $N(\varepsilon_{2q})=1$. Then, using    Lemma \ref{a5} and the same technique as above  we show that there are two integers $d_1$ and $d_2$ such that
 		$ \left\{ \begin{array}{ll}
 		c-1=d_1^2\\
 		c+1=2qd_2^2.
 		\end{array}\right. $ Thus,  $\sqrt{ 2\varepsilon_{  2q}}=d_1 +d_2\sqrt{2q}$ and 	$2=-d_1^2+2qd_2^2$.
 		
 		\item As  $N(\varepsilon_{q})=1$. Then, using Lemma \ref{a5} and the same technique as above  we show that there are two integers $\beta_1$ and $\beta_2$ such that
 		$ \left\{ \begin{array}{ll}
 		\alpha-1=\beta_1^2\\
 		\alpha+1=q\beta_2^2.
 		\end{array}\right. $ Thus,  $\sqrt{ 2\varepsilon_{  q}}=\beta_1 +\beta_2\sqrt{q}$ and 	$2=-\beta_1^2+q\beta_2^2$.
 		
 	\end{enumerate}
 \end{proof}

 \begin{corollary}\label{Corr units of biquad under condi (1)}
 	Let $p$ and $q$ be two primes satisfying    conditions $(\ref{cond 1})$.
 	\begin{enumerate}[\rm 1.]
 		\item  A FSU of $ \mathbb{Q}(\sqrt{p},\sqrt{q})$ is given by $\{\varepsilon_{p}, \varepsilon_{q},
 		\sqrt{ \varepsilon_{pq}}\}$.
 		\item A FSU of $\mathbb{Q}(\sqrt{2},\sqrt{q})$ is given by $\{\varepsilon_{2}, \sqrt{\varepsilon_{q}},
 		\sqrt{ \varepsilon_{2q}}\}$.
 		\item A FSU of  $\mathbb{Q}(\sqrt p, \sqrt{2q})$ is  given by $\{\varepsilon_{p}, \varepsilon_{2q},
 		\sqrt{ \varepsilon_{2q}\varepsilon_{2pq}}\}$.
 		\item A FSU of  $\mathbb{Q}(\sqrt q, \sqrt{2p})$ is  given by  $\{\varepsilon_{q}, \varepsilon_{2p},
 		\sqrt{\varepsilon_{2pq}}\}$.
 		\item A FSU of  $\mathbb{Q}(\sqrt 2, \sqrt{pq})$  is  given by $\{\varepsilon_{2}, \varepsilon_{pq},
 		\sqrt{\varepsilon_{pq}\varepsilon_{2pq}}\}$.
 	\end{enumerate}
 	
 \end{corollary}	
 \begin{proof}
 	Note that $\sqrt{2}\not\in \mathbb{Q}(\sqrt{p},\sqrt{q})$ and $\varepsilon_{p}$ has a negative norm. So using Lemma \ref{lm expressions of units under cond 1},
 	one easily verifies that the only element of the form $\varepsilon_{pq}^i\varepsilon_{p}^j\varepsilon_{q}^k$, for $i, j$ and $k\in \{0,1\}$, which is a square in $ \mathbb{Q}(\sqrt{p},\sqrt{q})$,  is $\varepsilon_{pq}$. So the first item by the method given in Page \pageref{algo wada}. One can similarly deduce the rest from  Lemma \ref{lm expressions of units under cond 1} and \cite[Propositions  3.1 and 3.2]{AZT2016}.
 \end{proof}

Now we are able to state the first important result of this section.

 	\begin{theorem}\label{thm first main  thm on units under conditions (1)}Let $p$ and $q$ be two primes satisfying conditions $(\ref{cond 1})$.
 	Put $\KK=\QQ(\sqrt 2, \sqrt{p}, \sqrt{q}, \sqrt{-1} )$ and  $\KK^+=\QQ(\sqrt 2, \sqrt{p}, \sqrt{q} )$. Then
 	\begin{enumerate}[\rm 1.]
 		\item
 		\begin{enumerate}[\rm a.]
 			\item $E_{\KK^+}=\langle-1,  \varepsilon_{2}, \varepsilon_{p},  \sqrt{\varepsilon_{q}}, \sqrt{\varepsilon_{2q}},
 			\sqrt{ \varepsilon_{pq}}, \sqrt{\varepsilon_{2}\varepsilon_{p}\varepsilon_{2p}},
 			\sqrt[4]{   \varepsilon_{p}^2 {\varepsilon_{2q}}{\varepsilon_{pq}\varepsilon_{2pq}}   } \rangle.$
 			\item The class number of $\KK^+$ is odd.
 		\end{enumerate}
 		\item \begin{enumerate}[\rm a.]
 			\item$E_{\KK}=\langle \zeta_{24}\text{ or } \zeta_8 ,  \varepsilon_{2}, \varepsilon_{p},  \sqrt{\varepsilon_{q}},
 		\sqrt{ \varepsilon_{pq}}, \sqrt{\varepsilon_{2}\varepsilon_{p}\varepsilon_{2p}}, \sqrt[4]{   \varepsilon_{p}^2 {\varepsilon_{2q}}{\varepsilon_{pq}\varepsilon_{2pq}}     }, \sqrt[4]{   \zeta_8^2\varepsilon_{2}^2 {\varepsilon_{q}}{\varepsilon_{2q}} }\rangle$,
 		according to whether $q=3$ or not.
 			\item 	$h_2(\KK)=h_2(-pq)$.
 	\end{enumerate}
 	\end{enumerate}
 \end{theorem}
 \begin{proof}
 	\begin{enumerate}[\rm 1.]
 		\item  %	To prove this theorem, we use the methods given by \cite{wada, azizunints99} . %we use the algorithm described by \cite{wada}.
 		Consider the  following   diagram (see Figure \ref{fig:1}):
 		\begin{figure}[H]
 			$$\xymatrix@R=0.8cm@C=0.3cm{
 				&\KK^+=\QQ( \sqrt 2, \sqrt{p}, \sqrt{q})\ar@{<-}[d] \ar@{<-}[dr] \ar@{<-}[ld] \\
 				k_1=\QQ(\sqrt 2,\sqrt{p})\ar@{<-}[dr]& k_2=\QQ(\sqrt 2, \sqrt{q}) \ar@{<-}[d]& k_3=\QQ(\sqrt 2, \sqrt{pq})\ar@{<-}[ld]\\
 				&\QQ(\sqrt 2)}$$
 			\caption{Subfields of $\KK^+/\QQ(\sqrt 2)$}\label{fig:1}
 		\end{figure}
 		
 		\noindent Note that by \cite[Théorème 6]{azizitalbi},     $\{\varepsilon_{2}, \varepsilon_{p},
 		\sqrt{\varepsilon_{2}\varepsilon_{p}\varepsilon_{2p}}\}$, is a FSU of $k_1$.
 		By Corollary \ref{Corr units of biquad under condi (1)},
 		  $\{\varepsilon_{2}, \sqrt{\varepsilon_{q}}, \sqrt{\varepsilon_{2q}}\}$  is a FSU of $k_2$ and a FSU of $k_3$ is given by $\{ \varepsilon_{2}, \varepsilon_{pq},
 		\sqrt{\varepsilon_{2pq}\varepsilon_{pq}}\}$. 	
 		
 		It follows that,  	$$E_{k_1}E_{k_2}E_{k_3}=\langle-1,  \varepsilon_{2}, \varepsilon_{p}, \varepsilon_{pq}, \sqrt{\varepsilon_{q}}, \sqrt{\varepsilon_{2q}},  \sqrt{\varepsilon_{pq}\varepsilon_{2pq}}, \sqrt{\varepsilon_{2}\varepsilon_{p}\varepsilon_{2p}}\rangle.$$	
 		Note that a FSU of $\KK$ consists  of seven  units chosen from those of $k_1$, $k_2$ and $k_3$, and  from the square roots of the units of $E_{k_1}E_{k_2}E_{k_3}$ which are squares in $\KK$ (cf. Page \pageref{algo wada}).
 		Thus we shall determine elements of $E_{k_1}E_{k_2}E_{k_3}$ which are squares in $\KK^+$. Suppose  $X$ is an element of $\KK^+$ which is the  square root of an element of $E_{k_1}E_{k_2}E_{k_3}$. We can assume that
 		$$X^2=\varepsilon_{2}^a\varepsilon_{p}^b \varepsilon_{pq}^c\sqrt{\varepsilon_{q}}^d\sqrt{\varepsilon_{2q}}^e\sqrt{\varepsilon_{pq}\varepsilon_{2pq}}^f
 		\sqrt{\varepsilon_{2}\varepsilon_{p}\varepsilon_{2p}}^g,$$
 		where $a, b, c, d, e, f$ and $g$ are in $\{0, 1\}$.
 		
 		We shall use norm maps from $\KK^+$ to its  subextensions  to eliminate  the cases of $X^2$ which do not occur.
 		% Set
 		%	$\mathrm{Gal}(\KK^+/\QQ)=\langle \tau_1, \tau_2, \tau_3\rangle$,
 		%	where	
 		%	\begin{center} $\tau_1(\sqrt{2})=-\sqrt{2}$, $\tau_2(\sqrt{p})=-\sqrt{p}$ and $\tau_3(\sqrt{q})=-\sqrt{q}$,\\ and
 		%	$\tau_i(\sqrt{2})=\sqrt{2}$ for $i\in\{2, 3\}$,\\
 		%	$\tau_i(\sqrt{p})=\sqrt{p}$ for $i\in\{1, 3\}$ and \\
 		%	$\tau_i(\sqrt{q})=\sqrt{q}$ for $i\in\{1, 2\}$.\end{center} 	
 		% \begin{center}	
 		Let $\tau_1$, $\tau_2$ and $\tau_3$ be the elements of  $ \mathrm{Gal}(\KK^+/\QQ)$ defined by
 		\begin{center}	\begin{tabular}{l l l }
 				$\tau_1(\sqrt{2})=-\sqrt{2}$, \qquad & $\tau_1(\sqrt{p})=\sqrt{p}$, \qquad & $\tau_1(\sqrt{q})=\sqrt{q},$\\
 				$\tau_2(\sqrt{2})=\sqrt{2}$, \qquad & $\tau_2(\sqrt{p})=-\sqrt{p}$, \qquad &  $\tau_2(\sqrt{q})=\sqrt{q},$\\
 				$\tau_3(\sqrt{2})=\sqrt{2}$, \qquad &$\tau_3(\sqrt{p})=\sqrt{p}$, \qquad & $\tau_3(\sqrt{q})=-\sqrt{q}.$
 			\end{tabular}
 		\end{center}
 		Note that  $\mathrm{Gal}(\KK^+/\QQ)=\langle \tau_1, \tau_2, \tau_3\rangle$
 		and the subfields  $k_1$, $k_2$ and $k_3$ are
 		fixed by  $\langle \tau_3\rangle$, $\langle\tau_2\rangle$ and $\langle\tau_2\tau_3\rangle$ respectively.
 	 	Lemma \ref{lm expressions of units under cond 1} is used to compute the norm maps from $\KK^+$ to its subextensions. We summarize these computations in     Table \ref{table1}.
 	Let us start	by applying   the norm map $N_{\KK^+/k_2}=1+\tau_2$. 	
 		\begin{eqnarray*}
 			N_{\KK^+/k_2}(X^2)=N_{\KK^+/k_2}(X)^2&=&
 			\varepsilon_{2}^{2a}(-1)^b \cdot 1\cdot \varepsilon_{q}^d\varepsilon_{2q}^e\cdot(-1)^f \cdot (-1)^{gv}\varepsilon_{2}^g\\
 			&=&	\varepsilon_{2}^{2a}  \varepsilon_{q}^d\varepsilon_{2q}^e\cdot(-1)^{b+f+gv} \varepsilon_{2}^g.
 		\end{eqnarray*}
 		%	$$N_{\KK^+/k_2}(X^2)=N_{\KK^+/k_2}(X)^2=
 		%	\varepsilon_{2}^{2a}(-1)^b \cdot 1\cdot \varepsilon_{q}^d\varepsilon_{2q}^e\cdot(-1)^f \cdot (-1)^{gv}\varepsilon_{2}^g
 		%	=	\varepsilon_{2}^{2a}  \varepsilon_{q}^d\varepsilon_{2q}^e\cdot(-1)^{b+f+gv} \varepsilon_{2}^g>0.$$
 		Note that by Corollary \ref{Corr units of biquad under condi (1)}, $\{\varepsilon_{2}, \sqrt{\varepsilon_{q}},
 		\sqrt{ \varepsilon_{2q}}\}$ is a FSU of $k_2$.	Thus  $\varepsilon_{q}$ and  $\varepsilon_{2q}$ are squares in $k_2$ whereas $\varepsilon_{2}$ is not.  Since  $N_{\KK^+/k_2}(X^2)>0$,  then  $b+f+vg\equiv0\pmod2$ and $\varepsilon_{2}^g$ is a square in $k_2$. Therefore    $g=0$ and   $b=f$. So we have
 		$$X^2=\varepsilon_{2}^a\varepsilon_{p}^f \varepsilon_{pq}^c\sqrt{\varepsilon_{q}}^d\sqrt{\varepsilon_{2q}}^e\sqrt{\varepsilon_{pq}\varepsilon_{2pq}}^f.$$
 		Similarly,  by  applying   $N_{\KK^+/k_3}=1+\tau_2\tau_3$ one gets:
 		\begin{eqnarray*}
 			N_{\KK^+/k_3}(X^2)&=&\varepsilon_{2}^{2a}\cdot (-1)^f\cdot \varepsilon_{pq}^{2c}\cdot (-1)^d  \cdot (-1)^e\cdot\varepsilon_{pq}^f \varepsilon_{2pq}^f \\
 			&=&\varepsilon_{2}^{2a}\varepsilon_{pq}^{2c}\varepsilon_{pq}^f\varepsilon_{2pq}^f(-1)^{f+d+e}>0.
 		\end{eqnarray*}
 		%$$N_{\KK^+/k_3}(X^2)=\varepsilon_{2}^{2a}\varepsilon_{2pq}^{2c}\varepsilon_{pq}^f\varepsilon_{2pq}^f(-1)^{f+d+e}>0, $$
 		Note that by Corollary \ref{Corr units of biquad under condi (1)}, $\varepsilon_{pq}\varepsilon_{2pq}$ is a square in $k_3$. Thus  all what we can deduce is   $f+d+e\equiv0\pmod2$. Let us now apply    $N_{\KK^+/k_4}=1+\tau_1$, where $k_4=\QQ(\sqrt p,  \sqrt q)$.
 		%	Note that $\{\varepsilon_{p},  \varepsilon_{q},  \sqrt{\varepsilon_{q}\varepsilon_{pq}}\}$  is a FSU of $k_4$.
 		We have
 		%for $$X^2=\varepsilon_{2}^a \varepsilon_{pq}^c\sqrt{\varepsilon_{q}}^d\sqrt{\varepsilon_{2q}}^e\sqrt{\varepsilon_{pq}\varepsilon_{2pq}}^f,$$ we get
 		\begin{eqnarray*}
 			N_{\KK^+/k_4}(X^2)&=&(-1)^{a}\cdot \varepsilon_{p}^{2f}\cdot \varepsilon_{pq}^{2c} \cdot (-\varepsilon_{q})^d\cdot 1 \cdot(\varepsilon_{pq})^f\\
 			&=&
 			\varepsilon_{p}^{2f} \varepsilon_{pq}^{2c}\varepsilon_{pq}^f\cdot (-1)^{a+d}\cdot \varepsilon_{q}^d>0.%\\
 			%	&=&(-1)^{a+c+e}		\varepsilon_{q}^d\varepsilon_{pq}^f>0.
 		\end{eqnarray*}
 		%$$N_{\KK^+/k_4}(X^2)=(-1)^{a}\varepsilon_{pq}^{2c}(-\varepsilon_{q})^d(-1)^e(\varepsilon_{pq})^f=(-1)^{a+c+e}
 		%\varepsilon_{pq}^{2c}\varepsilon_{q}^d\varepsilon_{pq}^f=(-1)^{a+c+e}
 		%\varepsilon_{q}^d\varepsilon_{pq}^f>0.$$
 		
 		\noindent Thus,  $a+d\equiv0\pmod2$. by Corollary \ref{Corr units of biquad under condi (1)} $\varepsilon_{pq}$ is a square in $k_4$ and by Lemma \ref{lm expressions of units under cond 1}, $2\varepsilon_{q}$ is a square in   $k_4$ whereas  $\varepsilon_{q}$ is not (in fact $\sqrt{2}\notin k_4$).
 		So $d=0$ and then $a=0$. Since  $f+d+e\equiv0\pmod2$, we have $f=e$. Therefore,
 		$$X^2=\varepsilon_{p}^f \varepsilon_{pq}^c\sqrt{\varepsilon_{2q}}^f\sqrt{\varepsilon_{pq}\varepsilon_{2pq}}^f.$$
 		Note that,  by Lemma \ref{lm expressions of units under cond 1},  $\varepsilon_{pq}$ is a square in $\KK^+$,  so we may put
 		$$X^2=\varepsilon_{p}^f \sqrt{\varepsilon_{2q}}^f\sqrt{\varepsilon_{pq}\varepsilon_{2pq}}^f.$$
 		
 		Suppose that $f=0$. Then by the above discussions and Lemma \ref{Lemme azizi}, a FSU of   $ \KK^+$ is $$\{\varepsilon_{2}, \varepsilon_{p},   \sqrt{\varepsilon_{q}}, \sqrt{\varepsilon_{2q}},  \sqrt{\varepsilon_{pq} },\sqrt{ \varepsilon_{2pq}}, \sqrt{\varepsilon_{2}\varepsilon_{p}\varepsilon_{2p}}\}.$$
 		
 		Thus  $q(\KK^+)=2^5$.	We have $h_2(p)=h_2(q)=h_2(2q)=h_2(2)=1$ and
 		$h_2(2p)=h_2(pq)=h_2(2pq)=2$  (cf.   \cite[Corollaries 18.4, 19.7 and 19.8]{connor88})\label{real cn 1}
 		
 		$$\begin{array}{ll}
 		h_2(\KK^+)&=\frac{1}{2^{9}}q(\KK^+)  h_2(2) h_2(p) h_2(q)h_2(2p) h_2(2q)h(pq)  h_2(2pq) \\
 		&=\frac{1}{2^{9}}\cdot 2^5\cdot 1\cdot 1 \cdot 1  \cdot 2 \cdot 1 \cdot 2 \cdot 2, \\
 		&= \frac{1}{2},
 		\end{array}$$
 		which is absurd. Thus $f=1$ and then $q(\KK^+)=2^6$.	So we have the first item.

 		\item Keep the notations of Lemma \ref{Lemme azizi}.   Note that   the greatest
 		integer $n_0$ such that $\zeta_{2^{n_0}}$ is contained in $\KK$ equals $  3$. So   $\mu_{n_0}=\sqrt{2}$. So according to Lemma \ref{Lemme azizi}, we should find an element $Y$, if it exists,  which is   in $\KK^+$ such that
	$$Y^2=(2+\sqrt{2})\varepsilon_{2}^a\varepsilon_{p}^b\sqrt{\varepsilon_{q}}^c\sqrt{\varepsilon_{2q}}^d\sqrt{\varepsilon_{pq}}^e \sqrt{\varepsilon_{2}\varepsilon_{p}\varepsilon_{2p}}^f
 		\sqrt[4]{   \varepsilon_{p}^2 {\varepsilon_{2q}}{\varepsilon_{pq}\varepsilon_{2pq}}     }^g,$$ 
 		where $a, b, c, d, e, f$ and $g$ are in $\{0, 1\}$.  
 		 So firstly we shall use norm maps to  eliminate
 		some cases (see Table \ref{table1}).
 		\begin{enumerate}[\rm $\bullet$]
 			\item	
 			
 			We have $N_{\KK^+/k_2}=1+\tau_2$. So by applying $N_{\KK^+/k_2}$, we get
 			\begin{eqnarray*}
 				N_{\KK^+/k_2}(Y^2)&=&(2+\sqrt{2})^2\varepsilon_{2}^{2a}(-1)^b\varepsilon_{q}^c\varepsilon_{2q}^d\cdot (-1)^{e}\cdot (-1)^{fv}\varepsilon_{2}^f(-1)^{gs}\sqrt{\varepsilon_{2q}}^g, \\	
 				&=&(2+\sqrt{2})^2\varepsilon_{2}^{2a}\varepsilon_{q}^c\varepsilon_{2q}^d(-1)^{b+e+fv+gs}\varepsilon_{2}^{f}\sqrt{\varepsilon_{2q}}^g>0.
 			\end{eqnarray*}
 			
 			Thus,  $b+e+fv+gs=0\pmod 2$.  By Corollary \ref{Corr units of biquad under condi (1)},  $\{\varepsilon_{2},  \sqrt{\varepsilon_{q}},  \sqrt{\varepsilon_{2q}}\}$  is a $\mathrm{FSU}$  of $k_2$.
 			Since $\varepsilon_{2}$, $\sqrt{\varepsilon_{2q}}$ and
 			$\varepsilon_{2}\sqrt{\varepsilon_{2q}}$ are  not   squares in $k_2$, we have  $f=g=0$ and so $b=e$.  Therefore, 	
 			
 			$$Y^2=(2+\sqrt{2})\varepsilon_{2}^a\varepsilon_{p}^e\sqrt{\varepsilon_{q}}^c\sqrt{\varepsilon_{2q}}^d\sqrt{\varepsilon_{pq}}^e.$$
 			
 			We have $N_{\KK^+/k_4}=1+\tau_1$ with $k_4=\QQ(\sqrt p,  \sqrt q)$. So	
 			\begin{eqnarray*}
 				N_{\KK^+/k_4}(Y^2)&=&(4-2){(-1)}^{a}\varepsilon_{p}^{2e}{(-1)}^c\varepsilon_q^c\cdot 1 \cdot \varepsilon_{pq}^e, \\	
 				&=&\varepsilon_{p}^{2e}\varepsilon_{pq}^e(-1)^{a+c}\cdot 2\cdot  \varepsilon_q^{c} >0.
 			\end{eqnarray*}
 			So $a+c=0\pmod 2$. Since $\sqrt{2}\notin k_4$ and by Lemma \ref{lm expressions of units under cond 1} $\sqrt{2\varepsilon_q}\in k_4$, then  $c=1$.	Therefore $a=c=1$ and we have
 			
 			$$Y^2=(2+\sqrt{2})\varepsilon_{2}\varepsilon_{p}^e\sqrt{\varepsilon_{q}}\sqrt{\varepsilon_{2q}}^d\sqrt{\varepsilon_{pq}}^e. $$
 			
 			%By applying the norm map,  $N_{\KK^+/k_5}=1+\tau_1\tau_3$ with $k_5=\QQ(\sqrt p,\sqrt{2q})$,  we get
 			%	$$N_{\KK^+/k_5}(Y^2)=(4-2).(-1)\cdot 1\cdot (-1)^d\varepsilon_{2q}^d\cdot (-1)^e\varepsilon_{2pq}^e= (-1)^{1+d+e}2\varepsilon_{2q}^d\varepsilon_{2pq}^e >0.$$
 			%	Thus $1+d+e=0\pmod 2$. Thus, $e\not d$.
 			By applying the norm map,  $N_{\KK^+/k_3}=1+\tau_2\tau_3$,  we get
 			\begin{eqnarray*}
 				N_{\KK^+/k_3}(Y^2)&=&(2+\sqrt{2})^2 \varepsilon_{2}^2\cdot (-1)^e\cdot (-1) \cdot (-1)^d.(-1)^e.\varepsilon_{pq}^e,\\
 				&=&(2+\sqrt{2})^2\varepsilon_{2}^2\cdot(-1)^{1+d}\cdot \varepsilon_{pq}^e  >0.
 			\end{eqnarray*}
 			% $$N_{\KK^+/k_3}(Y^2)=(2+\sqrt{2})^2 \varepsilon_{2}^2\cdot (-1) \cdot (-1)^d.(-1)^e.\varepsilon_{2pq}^e=(2+\sqrt{2})^2\varepsilon_{2}^2\cdot(-1)^{1+d+e}\cdot \varepsilon_{2pq}^e  >0.$$
 			Thus $1+d=0\pmod 2$. So   $d = 1$. As, by Corollary \ref{Corr units of biquad under condi (1)}, $\varepsilon_{pq}$ is a not a square in $k_3$, then $e=0$. It follows that
 			$$Y^2=(2+\sqrt{2})\varepsilon_{2}\sqrt{\varepsilon_{q}}\sqrt{\varepsilon_{2q}}. $$
 			Let us now verify that $(2+\sqrt{2})\varepsilon_{2}\sqrt{\varepsilon_{q}}\sqrt{\varepsilon_{2q}}$ is a square in $\KK^+$.
 			\item
 			Note that by  \cite[Theorem 5.5]{chemsZkhnin1}, the $2$-class group of  $L_{pq}:=\mathbb{Q}(\sqrt{pq},\sqrt{2},i) $ is cyclic. Since $\KK$ is an unramified quadratic extension of
 			$L_{pq} $, this implies that the Hilbert $2$-class field of $L_{pq} $(i.e.,  $L_{pq}^{(1)} $) and $\KK$ have the same Hilbert $2$-class field. So $h_2(L_{pq})=2h_2(\KK)$. Therefore, again by \cite[Lemma 3]{chemskatharina}, we have $2h_2(\KK)=2h_2(-pq)$. Thus, %$h_2(\KK)=h_2(-2q)$.
 			\begin{eqnarray}\label{proof eq condi 1}
 			h_2(\KK)=h_2(-pq)
 			\end{eqnarray}
 			Assume that $(2+\sqrt{2})\varepsilon_{2}\sqrt{\varepsilon_{q}}\sqrt{\varepsilon_{2q}}$ is not a square in $\KK^+$.
 			Then,
 			by Lemma \ref{Lemme azizi} and the above discussions, $\KK^+$ and $\KK$ have the same fundamental system of units. Thus $q(\KK)=2^7$.
 			 We have $h_2(-1)=h_2(-2)=h_2(-q)=1$, $h_2(-p)=h_2(-2p)=h_2(-2q)=2$ and $h_2(-2pq)=4$ by  \cite[Corollary 18.4]{connor88}, \cite[Corollary 19.6]{connor88}  and \cite[p. 353]{kaplan76} respectively.  So  by  class number formula (cf. \cite[p. 201]{wada}) and the above setting on the $2$-class numbers of real quadratic fields (Page \pageref{real cn 1})   we get\\
 			$\begin{array}{ll}
 			h_2(\KK)&=\frac{1}{2^{16}}q(\KK) h_2(-1) h_2(2) h_2(-2) h_2(p) h_2(-p) h_2(q) h_2(-q) h_2(2p)\\
 			&\qquad h_2(-2p) h_2(2q)h_2(-2q) h_2(pq) h_2(-pq) h_2(2pq) h_2(-2pq)\\
 			&=\;\frac{1}{2^{16}}\cdot 2^7\cdot 1\cdot 1 \cdot 1 \cdot 1 \cdot 2 \cdot 1 \cdot 1 \cdot 2 \cdot 2\cdot 1\cdot 2\cdot 2\cdot  h_2(-pq)\cdot 2\cdot 4, \\
 			&= \frac{1}{2}h_2(-pq).
 			\end{array}$\\
 			Which contradicts equation  \eqref{proof eq condi 1}. It follows that
 			$(2+\sqrt{2})\varepsilon_{2}\sqrt{\varepsilon_{q}}\sqrt{\varepsilon_{2q}}$ is a square in $\KK^+$. Hence Lemma \ref{Lemme azizi}   completes the proof.
 		\end{enumerate}
 	\end{enumerate}
 	
 \end{proof}

 	\newpage
 {\begin{table}[H]
 		\renewcommand{\arraystretch}{2.5}
 		\footnotesize
 		
 		\begin{center}\rotatebox{-90}{%We construct the following table:\\
 				\begin{tabular}{|c|c|c|c|c|c|c|c|c|c|}
 					\hline
 					$\varepsilon$ & $\varepsilon^{\tau_1}$ & $\varepsilon^{\tau_2}$ & $\varepsilon^{\tau_3}$ &$\varepsilon^{1+\tau_1}$ & $\varepsilon^{1+\tau_2}$ & $\varepsilon^{1+\tau_3}$ & $\varepsilon^{1+\tau_1\tau_2}$& $\varepsilon^{1+\tau_1\tau_3}$& $\varepsilon^{1+\tau_2\tau_3}$  \\ \hline
 					$\varepsilon_{2}$ & $\frac{-1}{\varepsilon_{2}}$ & $\varepsilon_{2}$ & $\varepsilon_{2}$ & $-1$ & $\varepsilon_{2}^2$ &$\varepsilon_{2}^2$&$-1$ & $-1$ & $\varepsilon_{2}^2$\\ \hline
 					
 					$\varepsilon_{p}$ & $\varepsilon_{p}$ & $\frac{-1}{\varepsilon_{p}}$ & $\varepsilon_{p}$ &$\varepsilon_{p}^2$ &$-1$ & $\varepsilon_{p}^2$ &$-1$& $\varepsilon_{p}^2$ &$-1$\\ \hline
 					%$\varepsilon_{pq}$ & $\varepsilon_{pq}$ & $\frac{1}{\varepsilon_{pq}}$ & $\frac{1}{\varepsilon_{pq}}$ & $\varepsilon_{pq}^2$& $1$ & $1$& $1$& $1$ &$\varepsilon_{pq}^2$\\ \hline
 					$\sqrt{\varepsilon_{q}}$ & $-\sqrt{\varepsilon_{q}}$ & $\sqrt{\varepsilon_{q}}$ & $\frac{-1}{\sqrt{\varepsilon_{q}}}$ & $-\varepsilon_{q}$ & $\varepsilon_{q}$ &$-1$&$-\varepsilon_{q}$&$1$&$-1$\\ \hline
 					$\sqrt{\varepsilon_{2q}}$ & $\frac{1}{\sqrt{\varepsilon_{2q}}}$ & $\sqrt{\varepsilon_{2q}}$ & $\frac{-1}{\sqrt{\varepsilon_{2q}}}$ &$1$ & $\varepsilon_{2q}$ & $-1$ &$1$& $-\varepsilon_{2q}$& $-1$\\ \hline
 					
 					$\sqrt{\varepsilon_{pq}}$ &$\sqrt{\varepsilon_{pq}}$&$\frac{-1}{\sqrt{\varepsilon_{pq}}}$ &$\frac{1}{\sqrt{\varepsilon_{pq}}}$& ${\varepsilon_{pq}}$ &$-1$ & $1$ & $-1$& $1$ & $-\varepsilon_{pq}$ \\ \hline
 					
 					$\sqrt{\varepsilon_{2pq}}$ &$\frac{1}{\sqrt{\varepsilon_{2pq}}}$&$\frac{1}{\sqrt{\varepsilon_{2pq}}}$ &$\frac{-1}{\sqrt{\varepsilon_{2pq}}}$& $1$ &$1$ & $-1$ & $\varepsilon_{2pq}$& $-\varepsilon_{2pq}$ & $-\varepsilon_{2pq}$ \\ \hline
 					
 					$\sqrt{\varepsilon_{2}\varepsilon_{p}\varepsilon_{2p}}$ & $(-1)^u\sqrt{\frac{\varepsilon_{p}}{\varepsilon_{2}\varepsilon_{2p}}}$ &  $(-1)^v\sqrt{\frac{\varepsilon_{2}}{\varepsilon_{p}\varepsilon_{2p}}}$ &
 					$\sqrt{\varepsilon_{2}\varepsilon_{p}\varepsilon_{2p}}$ & $(-1)^u\varepsilon_{p}$&$(-1)^v\varepsilon_{2}$ & $\varepsilon_{2}\varepsilon_{p}\varepsilon_{2p}$ & && \\
 					\hline

 					$\sqrt[4]{\varepsilon_{p}^2\varepsilon_{2q}\varepsilon_{pq}\varepsilon_{2pq}}$ & $(-1)^r\sqrt[4]{\frac{\varepsilon_{p}^2\varepsilon_{pq}}{\varepsilon_{2q}\varepsilon_{pq}}}$ &  $(-1)^s\sqrt[4]{\frac{\varepsilon_{2q}}{\varepsilon_{p}^2\varepsilon_{pq}\varepsilon_{2pq}}}$ &
 					$(-1)^t\sqrt[4]{\frac{\varepsilon_{p}^2}{\varepsilon_{2q}\varepsilon_{pq}\varepsilon_{2pq}}}$ & $(-1)^r\varepsilon_{p}\sqrt{\varepsilon_{pq}}$
 					&$(-1)^s\sqrt{\varepsilon_{2q}}$ &
 					$(-1)^t\varepsilon_{p}$ &  && \\
 					\hline
 					
 				\end{tabular}
 				%	\caption{Image of units by $\tau_i$}
 			}\caption{Norms   when $p$ and $q$ satisfy conditions  $(\ref{cond 1})$ }\label{table1}	\end{center}
 \end{table} }

To prove our second main result of this section, we need  the following  Lemma  and Corollary.
\begin{lemma}\label{lm expressions of units under cond 2}
	Let $p$ and $q$ be two primes satisfying    conditions $(\ref{cond 2})$.
	\begin{enumerate}[\rm 1.]
		\item Let  $x$ and $y$   be two integers such that
		$ \varepsilon_{2pq}=x+y\sqrt{2pq}$. Then we have
		\begin{enumerate}[\rm i.]
			\item $2p(x-1)$ is a square in $\NN$,
			\item $\sqrt{2\varepsilon_{2pq}}=y_1\sqrt{2p}+y_2\sqrt{q}$ and 	$2=-2py_1^2+ qy_2^2 $, for some integers $y_1$ and $y_2$.
		\end{enumerate}
		% $p(x-1)$ is a square in $\NN$, $\sqrt{2\varepsilon_{2pq}}=y_1\sqrt{p}+y_2\sqrt{2q}$ and 	$2=2qy_2^2-py_1^2$, for some integers $y_1$ and $y_2$   such that $y=y_1y_2$.
		
		\item  Let    $a$ and $b$ be two integers such that
		$ \varepsilon_{pq}=a+b\sqrt{pq}$. Then  we have
		\begin{enumerate}[\rm i.]
			\item $p(a+1)$ is a square in $\NN$,
			\item   $\sqrt{ 2\varepsilon_{ pq}}=b_1\sqrt{p}+b_2\sqrt{q}$ and 	$2=pb_1^2-qb_2^2$, for some integers $b_1$ and $b_2$.
		\end{enumerate}
		% $2p(a+1)$ is a square in $\NN$, $b$ is even,  $\sqrt{ \varepsilon_{ pq}}=b_1\sqrt{p}+b_2\sqrt{p}$ and 	$1=pb_1^2-qb_2^2$, for some integers $b_1$ and $b_2$   such that $b=2b_1b_2$.
		
		\item  Let    $c $ and $d$ be two integers such that
		$ \varepsilon_{2q}=c +d\sqrt{2q}$. Then  we have
		\begin{enumerate}[\rm i.]
			\item   $c-1$ is a square in $\NN$,
			\item  $\sqrt{ 2\varepsilon_{  2q}}=d_1 +d_2\sqrt{2q}$ and 	$2=-d_1^2+2qd_2^2$, for some integers $d_1$ and $d_2$.
		\end{enumerate}

		\item  Let    $\alpha $ and $\beta$ be two integers such that
		$ \varepsilon_{q}=\alpha +\beta\sqrt{q}$. Then  we have
		\begin{enumerate}[\rm i.]
			\item   $\alpha-1$ is a square in $\NN$,
			\item  $\sqrt{ 2\varepsilon_{  q}}=\beta_1 +\beta_2\sqrt{q}$ and 	$2=-\beta_1^2+q\beta_2^2$, for some integers $\beta_1$ and $\beta_2$.
		\end{enumerate}
	\end{enumerate}
\end{lemma}	
\begin{proof}
	We proceed similarly as in the proof of Lemma \ref{lm expressions of units under cond 1}.
\end{proof}	

\begin{corollary}\label{Corr units of biquad under condi (2)}
	Let $p$ and $q$ be two primes satisfying    conditions $(\ref{cond 2})$.
	\begin{enumerate}[\rm 1.]
		\item  A FSU of $ \mathbb{Q}(\sqrt{p},\sqrt{q})$ is given by $\{\varepsilon_{p}, \varepsilon_{q},
		\sqrt{ \varepsilon_{q}\varepsilon_{pq}}\}$.
		\item A FSU of $\mathbb{Q}(\sqrt{2},\sqrt{q})$ is given by $\{\varepsilon_{2}, \sqrt{\varepsilon_{q}},
		\sqrt{ \varepsilon_{2q}}\}$.
		\item A FSU of  $\mathbb{Q}(\sqrt p, \sqrt{2q})$ is given by $\{\varepsilon_{p}, \varepsilon_{2q},
		\sqrt{ \varepsilon_{2pq}}\}$.
		\item A FSU of  $\mathbb{Q}(\sqrt q, \sqrt{2p})$ is given by $\{\varepsilon_{q}, \varepsilon_{2p},
		\sqrt{\varepsilon_{q}\varepsilon_{2pq}}\}$.
		\item A FSU of  $\mathbb{Q}(\sqrt 2, \sqrt{pq})$   is given by  $\{\varepsilon_{2}, \varepsilon_{pq},
		\sqrt{\varepsilon_{pq}\varepsilon_{2pq}}\}$.
	\end{enumerate}	
\end{corollary}	

\begin{proof}
	We proceed similarly as in the proof of  Corollary \ref{Corr units of biquad under condi (1)}.
\end{proof}	
We can now state and prove the second main theorem of this section.
\begin{theorem}\label{thm first main  thm on units under conditions (2)}Let $p$ and $q$ be two primes satisfying conditions $(\ref{cond 2})$.
	Put $\LL=\QQ(\sqrt 2, \sqrt{p}, \sqrt{q}, \sqrt{-1} )$ and  $\LL^+=\QQ(\sqrt 2, \sqrt{p}, \sqrt{q} )$. Then
	\begin{enumerate}[\rm 1.]
		\item
			\begin{enumerate}[\rm a.]
			\item $E_{\LL^+}=\langle-1,  \varepsilon_{2}, \varepsilon_{p},  \sqrt{\varepsilon_{q}}, \sqrt{\varepsilon_{2q}},
		\sqrt{ \varepsilon_{pq}}, \sqrt{\varepsilon_{2}\varepsilon_{p}\varepsilon_{2p}},
		\sqrt[4]{\varepsilon_{2}^2\varepsilon_{p}^2\varepsilon_{q}\varepsilon_{pq}\varepsilon_{2pq}     }   \rangle.$
	\item $h_2(\LL^+)=1.$
	\end{enumerate}
		\item
			\begin{enumerate}[\rm a.]
			\item$E_{\LL}=\langle \zeta_{24}\text{ or } \zeta_8 ,  \varepsilon_{2}, \varepsilon_{p},  \sqrt{\varepsilon_{q}},
		\sqrt{ \varepsilon_{pq}}, \sqrt{\varepsilon_{2}\varepsilon_{p}\varepsilon_{2p}}, \sqrt[4]{\varepsilon_{2}^2\varepsilon_{p}^2\varepsilon_{q}\varepsilon_{pq}\varepsilon_{2pq}     },\sqrt[4]{  \zeta_8^2 \varepsilon_{2}^2 \varepsilon_{q}\varepsilon_{2q}} \rangle$,
		according to whether $q=3$ or not.
			\item  	$h_2(\LL)=h_2(-pq)=2$.
	\end{enumerate}
	\end{enumerate}
\end{theorem}
\begin{proof}
	
	\begin{enumerate}[\rm 1.]
		\item  	 	We consider an analogous diagram as in Figure \ref{fig:1}. Note that by  \cite[Théorème 6]{azizitalbi} and Corollary \ref{Corr units of biquad under condi (2)}, a FSU of $k_1=\QQ(\sqrt 2,  \sqrt p)$ is given by $\{\varepsilon_{2}, \varepsilon_{p},
		\sqrt{\varepsilon_{2}\varepsilon_{p}\varepsilon_{2p}}\}$,
		a FSU of $k_2=\QQ(\sqrt 2,  \sqrt q)$ is given by $\{\varepsilon_{2}, \sqrt{\varepsilon_{q}}, \sqrt{\varepsilon_{2q}}\}$  and	
		a FSU of $k_3=\QQ(\sqrt 2,  \sqrt{pq})$ is given by $\{ \varepsilon_{2}, \varepsilon_{pq},
		\sqrt{\varepsilon_{2pq}\varepsilon_{pq}}\}$.
		It follows that  	$$E_{k_1}E_{k_2}E_{k_3}=\langle-1,  \varepsilon_{2}, \varepsilon_{p}, \varepsilon_{pq}, \sqrt{\varepsilon_{q}}, \sqrt{\varepsilon_{2q}},  \sqrt{\varepsilon_{pq}\varepsilon_{2pq}}, \sqrt{\varepsilon_{2}\varepsilon_{p}\varepsilon_{2p}}\rangle.$$	
		Note that by Lemma \ref{Corr units of biquad under condi (2)},  $ \varepsilon_{pq}$ is a square in $\LL^+$. So we shall find    elements $X$ of $\LL^+$, if they exist, such that
		$$X^2=\varepsilon_{2}^a\varepsilon_{p}^b  \sqrt{\varepsilon_{q}}^c\sqrt{\varepsilon_{2q}}^d\sqrt{\varepsilon_{pq}\varepsilon_{2pq}}^e
		\sqrt{\varepsilon_{2}\varepsilon_{p}\varepsilon_{2p}}^f,$$
		where $a, b, c, d, e $ and $f$ are in $\{0, 1\}$.
		Let us define $\tau_1$, $\tau_2$ and $\tau_3$ similarly as in the proof of Theorem \ref{thm first main  thm on units under conditions (1)}. We shall use Table $\ref{table2}$.

		By applying   the norm map $N_{\LL^+/k_2}=1+\tau_2$, where $k_2=\QQ(\sqrt 2,  \sqrt q)$, we get 	
		\begin{eqnarray*}
			N_{\LL^+/k_2}(X^2)&=&
			\varepsilon_{2}^{2a}(-1)^b \cdot \varepsilon_{q}^c\varepsilon_{2q}^d\cdot(-1)^e \cdot (-1)^{fv}\varepsilon_{2}^f\\
			&=&	\varepsilon_{2}^{2a}  \varepsilon_{q}^c\varepsilon_{2q}^d\cdot(-1)^{b+e+fv} \varepsilon_{2}^f>0.
		\end{eqnarray*}
		We have  $b+e+fv \equiv0\pmod2$. By Corollary \ref{Corr units of biquad under condi (2)}, the units   $\varepsilon_{q}$ and  $\varepsilon_{2q}$ are squares in $k_2$ whereas $\varepsilon_{2}$ is not.   Then     $f=0$ and   $b=e$. Therefore
		$$X^2=\varepsilon_{2}^a\varepsilon_{p}^b  \sqrt{\varepsilon_{q}}^c\sqrt{\varepsilon_{2q}}^d\sqrt{\varepsilon_{pq}\varepsilon_{2pq}}^b,$$
		$N_{\LL^+/k_4}=1+\tau_1$, where $k_4=\QQ(\sqrt p,  \sqrt q)$.
		We have
		\begin{eqnarray*}
			N_{\LL^+/k_4}(X^2)&=&(-1)^{a}\cdot \varepsilon_{p}^{2b}\cdot  (- 1)^c  \cdot  \varepsilon_{q}^c\cdot 1 \cdot\varepsilon_{pq}^b\\
			&=&
			\varepsilon_{p}^{2b}   \cdot (-1)^{a+c}\cdot \varepsilon_{q}^c\varepsilon_{pq}^b>0.%\\
			%	&=&(-1)^{a+c+e}		\varepsilon_{q}^d\varepsilon_{pq}^f>0.
		\end{eqnarray*}
		Then $a+c=0\pmod 2$, so $a=c$. Since by Corollary  \ref{Corr units of biquad under condi (2)},	the units   $\varepsilon_{q}$ and
		$\varepsilon_{pq}$ are not squares in $k_4$, we have  $c=b$. Thus $a=b=c$ and
			$$X^2=\varepsilon_{2}^a\varepsilon_{p}^a  \sqrt{\varepsilon_{q}}^a\sqrt{\varepsilon_{2q}}^d\sqrt{\varepsilon_{pq}\varepsilon_{2pq}}^a,$$
		Let us now apply  $N_{\LL^+/k_3}=1+\tau_2\tau_3$, where  $k_3=\QQ(\sqrt 2,  \sqrt{pq})$. Then
			\begin{eqnarray*}
			N_{\LL^+/k_4}(X^2)= \varepsilon_{2}^{2a}\cdot(-1)^{a}\cdot  (- 1)^a  \cdot   (- 1)^d \cdot  \varepsilon_{pq}^a\varepsilon_{2pq}^a=
			\varepsilon_{2}^{2a}   \cdot   (- 1)^d \cdot  \varepsilon_{pq}^a\varepsilon_{2pq}^a>0. 	
		\end{eqnarray*}
		Thus, $d=0$. Hence, $X^2=\varepsilon_{2}^a\varepsilon_{p}^a  \sqrt{\varepsilon_{q}}^a\sqrt{\varepsilon_{pq}\varepsilon_{2pq}}^a.$
		
			If we suppose that  $a=0$. Then a FSU of   $ \LL^+$ is $$\{\varepsilon_{2}, \varepsilon_{p},   \sqrt{\varepsilon_{q}}, \sqrt{\varepsilon_{2q}},  \sqrt{\varepsilon_{pq} },\sqrt{ \varepsilon_{2pq}}, \sqrt{\varepsilon_{2}\varepsilon_{p}\varepsilon_{2p}}\}.$$
		Thus  $q(\LL^+)=2^5$.	We have $h_2(p)=h_2(q)=h_2(2q)=h_2(2)=1$ and
		$h_2(2p)=h_2(pq)=h_2(2pq)=2$  (cf. \cite[Corollaries 18.4, 19.7 and 19.8]{connor88})\label{real cn 2}. So the class number formula (cf. \cite[p. 201]{wada}) gives
		$$\begin{array}{ll}
		h_2(\LL^+)&=\frac{1}{2^{9}}q(\LL^+)  h_2(2) h_2(p) h_2(q)h_2(2p) h_2(2q)h(pq)  h_2(2pq) \\
		&=\frac{1}{2^{9}}\cdot 2^5\cdot 1\cdot 1 \cdot 1  \cdot 2 \cdot 1 \cdot 2 \cdot 2, \\
		&= \frac{1}{2},
		\end{array}$$
		which is absurd. So necessarily  $a=1$ and then $q(\LL^+)=2^6$.	Therefore we have the first item.
		
		\item We shall proceed as in the second part of the proof of Theorem \ref{thm first main  thm on units under conditions (1)}. So 
		let
		$$Y^2=(2+\sqrt{2})\varepsilon_{2}^a\varepsilon_{p}^b\sqrt{\varepsilon_{q}}^c\sqrt{\varepsilon_{2q}}^d\sqrt{\varepsilon_{pq}}^e \sqrt{\varepsilon_{2}\varepsilon_{p}\varepsilon_{2p}}^f
		\sqrt[4]{\varepsilon_{2}^2\varepsilon_{p}^2\varepsilon_{q}\varepsilon_{pq}\varepsilon_{2pq}}^g,$$  for some  $a, b, c, d, e, f$ and $g$ are in $\{0, 1\}$.	According to Lemma \ref{Lemme azizi}, we should find an element $Y$, if it exists,  which is   in $\LL^+$. So firstly we shall use norm maps to  eliminate
		some cases (see Table \ref{table2}).
		\begin{enumerate}[\rm $\bullet$]
			\item			
			We have $N_{\LL^+/k_2}=1+\tau_2$. Note that by Corollary \ref{Corr units of biquad under condi (2)},  $\{\varepsilon_{2},  \sqrt{\varepsilon_{q}},  \sqrt{\varepsilon_{2q}}\}$  is a $\mathrm{FSU}$  of $k_2$. So we have
			\begin{eqnarray*}
				N_{\LL^+/k_2}(Y^2)&=&(2+\sqrt{2})^2\varepsilon_{2}^{2a}(-1)^b\varepsilon_{q}^c\cdot \varepsilon_{2q}^d\cdot (-1)^{e}\cdot (-1)^{fv}\varepsilon_{2}^f(-1)^{gs}\varepsilon_{2}^g\sqrt{\varepsilon_{2q}}^g, \\	
				&=&(2+\sqrt{2})^2\varepsilon_{2}^{2a}\varepsilon_{q}^c\varepsilon_{2q}^d(-1)^{b+e+fv+gs}\varepsilon_{2}^{f+g}\sqrt{\varepsilon_{2q}}^g>0.
			\end{eqnarray*}
			
			Thus,  $b+e+fv+gs=0\pmod 2$.   We have
			\begin{enumerate}[\rm $*$]
				\item $f=g=1$ is impossible. In fact $\sqrt{\varepsilon_{  q}}$ is not square in $k_2$.
				\item $f\not=g  $ is impossible. In fact $\varepsilon_{2}\sqrt{\varepsilon_{2q}}$ and  $\varepsilon_{2}$ are not squares in $k_2$.
			\end{enumerate}
			Thus  $f=g=0$ and  $b=e$.  It follows that, 	
					$$Y^2=(2+\sqrt{2})\varepsilon_{2}^a\varepsilon_{p}^b\sqrt{\varepsilon_{q}}^c\sqrt{\varepsilon_{2q}}^d\sqrt{\varepsilon_{pq}}^b.$$
				We have $N_{\LL^+/k_4}=1+\tau_1$ with $k_4=\QQ(\sqrt p,  \sqrt q)$. Note that by Corollary \ref{Corr units of biquad under condi (2)},   $\{\varepsilon_{p}, \varepsilon_{q},
			\sqrt{ \varepsilon_{q}\varepsilon_{pq}}\}$ is a FSU of $k_4$. We have	
			\begin{eqnarray*}
				N_{\LL^+/k_4}(Y^2)&=&(4-2){(-1)}^{a}\cdot \varepsilon_{p}^{2b}\cdot{(-1)}^c\cdot\varepsilon_q^c\cdot 1 \cdot (-1)^b\cdot \varepsilon_{pq}^b, \\	
				&=&\varepsilon_{p}^{2b}\cdot(-1)^{a+b+c}\cdot 2\cdot  \varepsilon_q^{c}\cdot\varepsilon_{pq}^b >0.
			\end{eqnarray*}
			So $a+b+c=0\pmod 2$. Since $\varepsilon_q\varepsilon_{pq}$ is a square in $k_4$ whereas $2$ is not, this implies $c\not= b$. So $a=1$ and we have
				$$Y^2=(2+\sqrt{2})\varepsilon_{2} \varepsilon_{p}^b\sqrt{\varepsilon_{q}}^c\sqrt{\varepsilon_{2q}}^d\sqrt{\varepsilon_{pq}}^b,$$
				with $c\not= b$.
			Let us now apply,  $N_{\LL^+/k_3}=1+\tau_2\tau_3$,  we get
			\begin{eqnarray*}
				N_{\LL^+/k_3}(Y^2)&=&(2+\sqrt{2})^2\cdot \varepsilon_{2}^2  \cdot (-1)^b \cdot (-1)^c\cdot  (-1)^d.(-1)^b.\varepsilon_{pq}^b,\\
				&=&(2+\sqrt{2})^2 \cdot \varepsilon_{2}^2 \cdot(-1)^{c+d}\cdot \varepsilon_{pq}^b  >0.
			\end{eqnarray*}		
			Then $c+d=0\pmod 2$ and so $c=d$. By Corollary \ref{Corr units of biquad under condi (2)}, $\varepsilon_{pq}$ is a not a square in $k_3$, thus $b=0$.
			Since $c\not= b$, we have $c=d=1$. Therefore
			$$Y^2=(2+\sqrt{2})   \varepsilon_{2}   \sqrt{\varepsilon_{q}}\sqrt{\varepsilon_{2q}},$$
			Let now verify that $(2+\sqrt{2})  \varepsilon_{2} \sqrt{\varepsilon_{q}}\sqrt{\varepsilon_{2q}}$ is a square in $\LL^+$.
			\item  As in the second part of the proof of Theorem \ref{thm first main  thm on units under conditions (1)}, we show that
			\begin{eqnarray}\label{proof eq}
			h_2(\LL)=h_2(-pq)=2.
			\end{eqnarray}
			Assume that $(2+\sqrt{2})  \varepsilon_{2} \sqrt{\varepsilon_{q}}\sqrt{\varepsilon_{2q}}$ is not a square in $\LL^+$.
			Then,
			by Lemma \ref{Lemme azizi} and the above discussions, $\LL^+$ and $\LL$ have the same fundamental system of units.  Therefore $q(\LL)=2^7$.
			We have $h_2(-1)=h_2(-2)=h_2(-q)=1$, $h_2(-p)=h_2(-2p)=h_2(-2q)=h_2(-pq)=2$ and $h_2(-2pq)=4$ by  \cite[Corollary 18.4]{connor88}, \cite[Corollary 19.6]{connor88}  and \cite[p. 353]{kaplan76} respectively.  So  by  class number formula (cf. \cite[p. 201]{wada}) and the above setting on the $2$-class numbers of real quadratic fields (Page \pageref{real cn 2})   we get \\
			$\begin{array}{ll}
			h_2(\LL)&=\frac{1}{2^{16}}q(\LL) h_2(-1) h_2(2) h_2(-2) h_2(p) h_2(-p) h_2(q) h_2(-q) h_2(2p)\\
			&\qquad h_2(-2p) h_2(2q)h_2(-2q) h_2(pq) h_2(-pq) h_2(2pq) h_2(-2pq)\\
			&=\;\frac{1}{2^{16}}\cdot 2^7\cdot 1\cdot 1 \cdot 1 \cdot 1 \cdot 2 \cdot 1 \cdot 1 \cdot 2 \cdot 2\cdot 1\cdot 2\cdot 2\cdot 2\cdot 2\cdot 4= 1.
			\end{array}$\\
			Which contradicts equation  \eqref{proof eq}. It follows that
			$(2+\sqrt{2})\varepsilon_{2}\sqrt{\varepsilon_{q}}\sqrt{\varepsilon_{2q}}$  is a square in $\LL^+$ and    so  Lemma \ref{Lemme azizi}   completes the proof.
		\end{enumerate}
	\end{enumerate}
	
\end{proof}

 	%\newpage
 {\begin{table}[H]
 		\renewcommand{\arraystretch}{2.5}
 		\footnotesize
 		
 		\begin{center}\rotatebox{-90}{%We construct the following table:\\
 				\begin{tabular}{|c|c|c|c|c|c|c|c|c|c|}
 					\hline
 					$\varepsilon$ & $\varepsilon^{\tau_1}$ & $\varepsilon^{\tau_2}$ & $\varepsilon^{\tau_3}$ &$\varepsilon^{1+\tau_1}$ & $\varepsilon^{1+\tau_2}$ & $\varepsilon^{1+\tau_3}$ & $\varepsilon^{1+\tau_1\tau_2}$& $\varepsilon^{1+\tau_1\tau_3}$& $\varepsilon^{1+\tau_2\tau_3}$  \\ \hline
 					$\varepsilon_{2}$ & $\frac{-1}{\varepsilon_{2}}$ & $\varepsilon_{2}$ & $\varepsilon_{2}$ & $-1$ & $\varepsilon_{2}^2$ &$\varepsilon_{2}^2$&$-1$ & $-1$ & $\varepsilon_{2}^2$\\ \hline
 					
 					$\varepsilon_{p}$ & $\varepsilon_{p}$ & $\frac{-1}{\varepsilon_{p}}$ & $\varepsilon_{p}$ &$\varepsilon_{p}^2$ &$-1$ & $\varepsilon_{p}^2$ &$-1$& $\varepsilon_{p}^2$ &$-1$\\ \hline
 					%$\varepsilon_{pq}$ & $\varepsilon_{pq}$ & $\frac{1}{\varepsilon_{pq}}$ & $\frac{1}{\varepsilon_{pq}}$ & $\varepsilon_{pq}^2$& $1$ & $1$& $1$& $1$ &$\varepsilon_{pq}^2$\\ \hline
 					$\sqrt{\varepsilon_{q}}$ & $-\sqrt{\varepsilon_{q}}$ & $\sqrt{\varepsilon_{q}}$ & $\frac{-1}{\sqrt{\varepsilon_{q}}}$ & $-\varepsilon_{q}$ & $\varepsilon_{q}$ &$-1$&$-\varepsilon_{q}$&$1$&$-1$\\ \hline
 					$\sqrt{\varepsilon_{2q}}$ & $\frac{1}{\sqrt{\varepsilon_{2q}}}$ & $\sqrt{\varepsilon_{2q}}$ & $\frac{-1}{\sqrt{\varepsilon_{2q}}}$ &$1$ & $\varepsilon_{2q}$ & $-1$ &$1$& $-\varepsilon_{2q}$& $-1$\\ \hline
 					
 					$\sqrt{\varepsilon_{pq}}$ &$-\sqrt{\varepsilon_{pq}}$&$\frac{-1}{\sqrt{\varepsilon_{pq}}}$ &$\frac{1}{\sqrt{\varepsilon_{pq}}}$& ${-\varepsilon_{pq}}$ &$-1$ & $1$ & $1$& $-1$ & $-\varepsilon_{pq}$ \\ \hline
 					
 					$\sqrt{\varepsilon_{2pq}}$ &$\frac{-1}{\sqrt{\varepsilon_{2pq}}}$&$\frac{1}{\sqrt{\varepsilon_{2pq}}}$ &$\frac{-1}{\sqrt{\varepsilon_{2pq}}}$& $-1$ &$1$ & $-1$ & $-\varepsilon_{2pq}$& $\varepsilon_{2pq}$ & $-\varepsilon_{2pq}$ \\ \hline
 					
 					$\sqrt{\varepsilon_{2}\varepsilon_{p}\varepsilon_{2p}}$ & $(-1)^u\sqrt{\frac{\varepsilon_{p}}{\varepsilon_{2}\varepsilon_{2p}}}$ &  $(-1)^v\sqrt{\frac{\varepsilon_{2}}{\varepsilon_{p}\varepsilon_{2p}}}$ &
 					$\sqrt{\varepsilon_{2}\varepsilon_{p}\varepsilon_{2p}}$ & $(-1)^u\varepsilon_{p}$&$(-1)^v\varepsilon_{2}$ & $\varepsilon_{2}\varepsilon_{p}\varepsilon_{2p}$ & && \\
 					\hline

 					$	\sqrt[4]{\varepsilon_{2}^2\varepsilon_{p}^2\varepsilon_{q}\varepsilon_{pq}\varepsilon_{2pq}     }$ & $(-1)^r\sqrt[4]{\frac{ \varepsilon_{p}^2\varepsilon_{q}\varepsilon_{pq}  }{\varepsilon_{2}^2  \varepsilon_{2pq} }}$ &  $(-1)^s\sqrt[4]{\frac{ \varepsilon_{2}^2\varepsilon_{q}  }{\varepsilon_{p}^2  \varepsilon_{pq}\varepsilon_{2pq} }}$ &
 					$(-1)^t\sqrt[4]{\frac{ \varepsilon_{2}^2\varepsilon_{p}^2  }{ \varepsilon_{q} \varepsilon_{pq}\varepsilon_{2pq} }}$ &  $(-1)^r	\varepsilon_{p}\sqrt {  \varepsilon_{q}\varepsilon_{pq}    }$
 					&$(-1)^s	\varepsilon_{2}\sqrt {  \varepsilon_{q}}    $ &
 					$(-1)^t	\varepsilon_{2}\varepsilon_{p}    $ &  && \\
 					\hline
 					
 				\end{tabular}
 				%	\caption{Image of units by $\tau_i$}
 			}\caption{Norms   when $p$ and $q$ satisfy conditions  $(\ref{cond 2})$  }\label{table2}	\end{center}
 \end{table} }

 \section{\textbf{  Remarks on the    Hilbert $2$-class field towers and the cyclotomic  $\mathbb{Z}_2$-extensions} }\label{section 3}

  Let $k$ be  an algebraic number field and   $\mathrm{Cl}_2(k)$ the $2$-Sylow subgroup of its ideal class group $\mathrm{Cl}(k)$. Let  $k^{(1)}$ (resp. $k^{(2)}$) be the first (resp. second) Hilbert $2$-class field of $k$ and  put $G=\mathrm{Gal}(k^{(2)}/k)$,  then if $G'$ denotes the commutator subgroup of $G$,  we have by class field theory $G'\simeq\mathrm{Gal}(k^{(2)}/k^{(1)})$ and  $G/G'\simeq\mathrm{Gal}(k^{(1)}/k)\simeq\mathrm{Cl}_2(k)$. Assume in all what follows that $\mathrm{Cl}_2(k)$ is of type $(2,  2)$,  then
  in \cite{kisilvsky}, Kisilevsky showed that
    $G$ is isomorphic to $\mathbb{Z}/2\mathbb{Z}\times \mathbb{Z}/2\mathbb{Z}  $,  $Q_m$,  $D_m$   or $S_m$, where $Q_m$,   $D_m$,  and $S_m$ denote  the quaternion,   dihedral and semidihedral groups respectively,
  of order $2^m$,  where $m\geq3$ and  $m\geq4$ for $S_m$. Let $F_1$,   $F_2$ and  $F_3$ be the three  unramified  quadratic extensions of $k$ and assume  that the $2$-class group of 
  $F_1$ is cyclic. Then using some known results of group theory
   one can easily deduce from \cite[Theorem 2]{kisilvsky} that we have the following Remark (cf. \cite[Remark 2.2]{chemszekhniniaziziUnits1}) :
   	\begin{remark}
   		The $2$-class groups of the two   fields $F_2 $  and $F_3$ are cyclic if and only if  $k^{(1)}=k^{(2)}$ or  $k^{(1)}\not=k^{(2)}$ and $G\simeq Q_3$. In the other cases the $2$-class groups  $F_2 $  and $F_3$ are of type $(2,  2)$ (whereas that of $F_1$ is cyclic).
   	\end{remark}
    Set the following notations:
   \begin{enumerate}
   	\item  $L_{pq}$: $\mathbb{Q}(\sqrt2,\sqrt{qp},i) $,
   	\item $F_{pq}$: $\QQ(\sqrt{p}, \sqrt{2q},i)$,
   	\item $K_{pq}$:  $\QQ(\sqrt{2p}, \sqrt{q},i)$.
   \end{enumerate}

\begin{remark}\label{ rmk kk}
		Let $p$ and $q$ be two primes    satisfying    conditions $(\ref{cond 1})$ or $(\ref{cond 2})$.	Note that by Lemmas \ref{lm expressions of units under cond 1} and \ref{lm expressions of units under cond 2}, $x\pm 1$ is not a square in $\NN$, where $x$ and $y$ are the  two integers such that $\varepsilon_{2pq}=x+y\sqrt{2pq}$. Thus,   the Hasse's unit index of  $\kk=\mathbb{Q}(\sqrt{2pq},i)$  equals $1$ (cf. \cite[3.(1)  p. 19]{azizunints99}). So by \cite{azizi99} the $2$-class group of $\kk$ is of type $(2,2)$. We similarly deduce, by using        Lemmas \ref{lm expressions of units under cond 1} and \ref{lm expressions of units under cond 2}, \cite[Lemma 4.1]{chemszekhniniaziziUnits1} and \cite[3.(1)  p. 19]{azizunints99},  that  the condition on Hasse's unit index of $\kk$, in \cite[Théorème 21]{Az-00},  is always verified (so in particular we are in condition of this theorem).
\end{remark}

 	\begin{theorem}\label{thm 2-class group of the 3 triquad}
 	Let $p$ and $q$ be two primes    satisfying   conditions $(\ref{cond 1})$. Then we have
 	\begin{enumerate}[\rm 1.]
 		\item The $2$-class group of $L_{pq}$ is $\ZZ/2^{m+1}\ZZ$, 	with $h_2(-pq)=2^m$.	
 		\item The $2$-class group of $F_{pq}$ is of type $(2,2)$.
 		\item The $2$-class group of $K_{pq}$ 	is of type $(2,2)$.
 	\end{enumerate}
 \end{theorem}
 \begin{proof}
Let $\kk=\mathbb{Q}(\sqrt{2pq},i)$. By Remark \ref{ rmk kk}, the $2$-class group of $\kk$ is of type $(2,2)$. Note that $L_{pq}$, $F_{pq}$ and $K_{pq}$ are the three unramified quadratic extensions of $\kk$.
Note that by  \cite[Theorem 10]{chemskatharina} the $2$-class group $L_{pq}$ is cyclic. So the above preliminaries complete  the proof.
 	
 \end{proof}

  Since the $2$-class group of $L_{pq}$ is cyclic,     the Hilbert $2$-class field tower of $L_{pq}$ terminates at the first layer.
Now we shall determine the structure of the groups $Gal(F_{pq}^{(2)}/F_{pq}) $ and  $Gal(K_{pq}^{(2)}/K_{pq}) $.
 	\begin{theorem}\label{thm second 2-class group of the 3 triquad}
 	Let $p$ and $q$ be two primes and $m$ such that and $h_2(-pq)=2^m$.
 	\begin{enumerate}[\rm 1.]
 		\item  	If  $p$ and $q$    satisfy   conditions \eqref{cond 1}, then we have
 		$$Gal(F_{pq}^{(2)}/F_{pq}) \simeq  Gal(K_{pq}^{(2)}/K_{pq}) \simeq Q_{m+1}.$$
 		
 	  \item    	If  $p$ and $q$      satisfy   conditions \eqref{cond 2}, then we have
 	  $$Gal(F_{pq}^{(2)}/F_{pq}) \simeq  Gal(K_{pq}^{(2)}/K_{pq}) \simeq \mathbb{Z}/4\mathbb{Z}.$$
 	\end{enumerate}
 \end{theorem}
 \begin{proof}
 
 Let   $L_{1,pq}= \mathbb{Q}(\sqrt2,\sqrt{q},\sqrt{p},i) $. Since   the $2$-class group of $L_{pq}$ is cyclic (so its Hilbert $2$-class field tower terminates at the first  layer) and $L_{1,pq}$ is an unramified extension of $L_{pq}$, we have the $2$-class group of $L_{1,pq}$ is also cyclic.
  As $L_{1,pq}$	is also a quadratic unramified extension of both 
  $F_{pq}$ and $K_{pq}$, then by \cite[Proposition 2.2]{chemszekhniniaziziUnits1} and   Theorems \ref{thm first main  thm on units under conditions (1)} and \ref{thm first main  thm on units under conditions (2)}   we have 
 $|Gal(F_{pq}^{(2)}/F_{pq})|=|Gal(K_{pq}^{(2)}/K_{pq})|=2\cdot h_2(L_1)=2\cdot  h_2(-pq)$. Since $h_2(-pq)$ is even, and $h_2(-pq)=2$ if and only if $\left(\dfrac{p}{q}\right)=-1$
  (cf. \cite[Corollaries 18.4 and 19.6]{connor88}), then we have the
 orders of the groups in question in   both cases. Suppose that 
$\left(\dfrac{p}{q}\right)=1$, then by \cite[Théorème 21]{Az-00}, Remark \ref{ rmk kk} and the above discussions,  the two groups in question
  are subgroups of index $2$ of   $Q_{m+2}$. So they are also quaternion of order $2^{m+1}$.
  
   Suppose now that  $\left(\dfrac{p}{q}\right)=-1$. Then, by \cite[Théorème 21]{Az-00}  and Remark \ref{ rmk kk} the groups in question are subgroups  of  $Q_3$ of index $2$.    So they are    cyclic. Which completes the proof.
 \end{proof}

\begin{remark}
	Put $\kk=\mathbb{Q}(\sqrt{2pq},i)$ and assume that $p$ and $q$ are two primes satisfying conditions \eqref{cond 1}. The author  of 
		\cite{Az-00} did not determine the order of $Gal(\kk^{(2)}/\kk)$. Now it is easy to deduce that it is of order $|Gal(\kk^{(2)}/\kk)|=2^{m+2}$ (i.e,
		$Gal(\kk^{(2)}/\kk)\simeq Q_{m+2}$), where $m$ is such that $h_2(-pq)=2^m$.
\end{remark}

 	\begin{theorem}
 	Let $p$ and $q$ be two primes satisfying conditions   \eqref{cond 1} or  \eqref{cond 2}. Put $\pi_1= 2$, $\pi_2=2+\sqrt{2}$,..., $\pi_n=2+\sqrt{\pi_{n-1}}$,
 	$L_n=\mathbb{Q}(\sqrt{q}, \sqrt{p}, \zeta_{2^{n+2}})$ and $L_n^+=\mathbb{Q}(\sqrt{p}, \sqrt{q}, \sqrt{\pi_n})$.
 	Then
 	\begin{enumerate}[\rm 1.]
 		\item For all $n\geq 1$,   the $2$-class group of  $L_n^+$ is trivial.
 		\item For all $n\geq 1$, the $2$-class group of  $L_n$ is $\ZZ/2^{n+m-1}\ZZ$, 	where $h_2(-pq)=2^m$.
 	\end{enumerate}
 \end{theorem}

 \begin{proof}
 	\begin{enumerate}[\rm 1.]
 		\item  	
 		We claim that the $2$-class group of 	 $k=\mathbb{Q}(\sqrt{p}, \sqrt{q})$ is trivial. In fact
 		by Corollaries \ref{Corr units of biquad under condi (1)} and \ref{Corr units of biquad under condi (2)},  \cite[Corollaries 18.4 and 19.7]{connor88} and  Kuroda's class number formula (\cite[p. 247]{lemmermeyer1994kuroda}), we obtain
 		$$h_2( k)=\frac{1}{4}q(k)h_2(p)h_2(q)h_2(pq)=\frac{1}{4}\cdot2\cdot1\cdot 1\cdot 2=1.$$
 		By Theorems \ref{thm first main  thm on units under conditions (1)} and \ref{thm first main  thm on units under conditions (2)}, the class number of $ \mathbb{Q}(\sqrt{p}, \sqrt{q}, \sqrt{2})$ the first step  of the cyclotomic
 		$\mathbb{Z}_2$-extension of $k$ is odd. So we have proved the first item by \cite[Theorem 1]{fukuda}.	
 		
 		\item  Note that $L_n$ is the genus field   of $L_{n,pq}=\mathbb{Q}(\sqrt{pq}, \zeta_{2^{n+2}})$ and $[L_n:L_{n,pq}]=2$.  By \cite[Theorem 10]{chemskatharina}, the $2$-class group of $L_{n,pq}$ is isomorphic to a cyclic group of order $2^{n+m}$, therefore that of $L_n$ is also cyclic and   $h_2(L_n)=\frac{h_2(L_{n,pq})}{2}=2^{n+m-1}$. So the second item is proved.	\end{enumerate}	
 \end{proof}
\begin{remark}	Let $p$ and $q$ be two primes satisfying conditions   \eqref{cond 1} or  \eqref{cond 2}. By the above theorem,
	the  Iwasawa invariants  $\lambda_2$ and $\nu_2$ of the fields $F_{pq}$, $K_{pq}$ and $\mathbb{Q}(\sqrt{p}, \sqrt{q},i)$ are  equal to $1$ and $m-1$ respectively.
\end{remark}

\section*{\textbf{Acknowledgments}}
The   author would like to thank his  professor Abdelkader  Zekhnini  for reading   the preliminary versions of this paper as well for   his   comments. Many thanks are also  due to      the referee for constructive comments which helped to improve this article.

\end{document}